\documentclass[aop,preprint]{imsart}

\RequirePackage[OT1]{fontenc}
\RequirePackage{amsthm,amsmath,amsfonts}
\RequirePackage[numbers]{natbib}
\RequirePackage[colorlinks,citecolor=blue,urlcolor=blue]{hyperref}
\usepackage[mathscr]{eucal}
\usepackage{amsfonts}
\usepackage{amsmath,amsxtra,latexsym,amsthm, amssymb, amscd}
\usepackage{graphicx}
\usepackage{color}
\startlocaldefs
\numberwithin{equation}{section}
\theoremstyle{plain}
\newtheorem{theorem}{Theorem}[section]
\newtheorem{lemma}[theorem]{Lemma}
\newtheorem{proposition}[theorem]{Proposition}

\theoremstyle{definition}
\newtheorem{definition}[theorem]{Definition}
\newtheorem{remark}[theorem]{Remark}

\endlocaldefs
\allowdisplaybreaks

\begin{document}

\begin{frontmatter}
\title{Strict solutions to  stochastic parabolic evolution equations in M-type 2 Banach spaces}
\runtitle{Stochastic parabolic evolution equations}

\begin{aug}
\author{T$\hat{\text{o}}$n Vi$\hat{\d{e}}$t T\d{a}\thanksref{t2}$,$ \ead[label=e1]{tavietton[at]agr.kyushu-u.ac.jp} Yoshitaka Yamamoto$,$ \ead[label=e2]{yamamoto[at]ist.osaka-u.ac.jp$$} Atsushi Yagi\ead[label=e3]{atsushi-yagi[at]ist.osaka-u.ac.jp}}
\runauthor{T\d{a}, Yamamoto \& Yagi}
\thankstext{t2}{This work   was supported by JSPS KAKENHI Grant Number 20140047.}
\affiliation{(Kyushu University and Osaka University, Japan)}

\address{T\d{a}, T.~V.\\
Center for Promotion of International Education and Research \\
  Faculty of Agriculture\\  
Kyushu University\\
6-10-1 Hakozaki, Higashi-ku, Fukuoka 812-8581, Japan\\
\printead{e1}\\
\\
Yamamoto, Y. \\
Department of Information and Physical Sciences\\
  Graduate School of Information Science
 and Technology \\
 Osaka University\\
  Suita, Osaka 565-0871, Japan\\
\printead{e2}\\
\\
Yagi, A. \\
Department of Applied Physics\\ 
Graduate School of Engineering\\
Osaka University\\
 Suita Osaka 565-0871, Japan\\
\printead{e3}}

\end{aug}

\begin{abstract}
{\bf Abstract.}  We study a  stochastic linear evolution equation  
$dX+A(t)Xdt=F(t)dt+ G(t)dw_t$ in  a Banach space of M-type 2. We construct  unique strict solutions to the equation on the basis of the theory of deterministic linear evolution equations. The abstract results are applied to stochastic diffusion equations.


\end{abstract}

%\begin{keyword}[class=MSC]
%\kwd[Primary ]{60H15}
%\kwd{35R60}
%\kwd[; secondary ]{47D06}
%\end{keyword}

%\begin{keyword}
%\kwd{Analytical semigroups}
%\kwd{Stochastic linear evolution equations}
%\kwd{Strict solutions}
%\kwd{M-type 2 Banach spaces}
%\end{keyword}
%\tableofcontents
\end{frontmatter}

\section {Introduction}
We consider the Cauchy problem for a     stochastic parabolic evolution equation
%%%%%%%%% E2
\begin{equation} \label{E0}
\begin{cases}
dX+A(t)Xdt=F(t)dt+ G(t)dw_t,\hspace{1cm} 0<t\leq T,\\
X(0)=\xi
\end{cases}
\end{equation}
in a  complex separable Banach space $E$ with  norm $\|\cdot\|$. 
Here, $\{w_t, t\geq 0\}$ denotes  a $\mathbb C^d$\,{-}\,valued  Brownian motion $(d=1,2,3,\dots)$ on a  complete probability space  $(\Omega, \mathcal F,\mathbb P)$ with filtration $\{\mathcal F_t\}_{t\geq 0}$ satisfying the usual conditions. 
 Given functions $F$  and $G$ are  $E$\,{-}\,valued and $L(\mathbb C^d; E)$\,{-}\,valued progressively measurable   processes on $[0,T]$, respectively, where $L(\mathbb C^d; E)$ denotes the space of all bounded linear operators from $\mathbb C^d$ into $E$. Initial value $ \xi$ is an $E$\,{-}\,valued $\mathcal  F_0$\,{-}\,measurable  random variable.  And,  $A(t), 0\leq t\leq T,$ are   densely defined, closed 
 linear operators in $E$ which generate an evolution operator on $E$.

Stochastic evolution equations  have  been  studied by many researchers. A comprehensive theory of abstract stochastic evolution equations in Hilbert spaces can be found in Da Prato-Zabczyk \cite{prato}.  In Banach setting, 
stochastic integrations and stochastic evolution equations have been constructed in M-type 2 Banach spaces  (see Brze\'{z}niak   \cite{Brzezniak} and Dettweiler \cite{Dettweiler}), and in UMD Banach spaces of type 2 (see van Neerven et al.  \cite{van1,van2}).  Using the stochastic integration in M-type 2 Banach spaces, stochastic differential equations on Banach manifolds and set-valued stochastic differential equations have also investigated (see Brze\'{z}niak-Elworthy  \cite{Brzezniak3}, Malinowski \cite{Malinowski} and  references therein).

In the studies on stochastic evolution equations, the main interest was to construct unique mild solutions.
 Some researchers were, however,  devoted to constructing stronger solutions, strict solutions.

 In 1992, Da Prato-Zabczyk \cite{prato} studied a   stochastic linear evolution equation 
%%%%%%%%% E2
\begin{equation} \label{E0.1}
\begin{cases}
dX+AXdt=F(t)dt+ I dW(t),\hspace{1cm} 0<t\leq T,\\
X(0)=\xi
\end{cases}
\end{equation}
in a   separable Hilbert space $H$. Here, $A$ is a densely defined, closed 
 linear operator in $H$ which generates an analytic semigroup on $H$, 
 $F$ is a predictable, integrable process on $[0,T]$, $W$  is a $Q$\,-\,Wiener process on  $H$ with   a  symmetric nonnegative nuclear operator $Q$ in $ L(H)$, and $I$ is the identity operator in $H$.
Under the  conditions:
\begin{itemize}
 \item [{\rm (i)}] 
  The trace of $  Q  $ is finite
  \item [{\rm (ii)}]  $A^\beta Q^{\frac{1}{2}} $ for some $ \frac{1}{2}<\beta\leq 1  $ is a Hilbert-Schmidt operator
\end{itemize}
 the authors proved  
 existence of strong solutions. Clearly, the condition  {\rm (ii)}  is very restrictive. In the case of standard Brownian motions in $\mathbb R^d$, $Q$ becomes  the identity matrix of size $d$.  So, {\rm (ii)} implies that $A$ is a bounded linear operator of $H$.

Brze\'{z}niak \cite{Brzezniak}  treated  a   stochastic linear evolution equation 
\begin{equation} \label{E0.5}
\begin{cases}
dX+AXdt = F(t)dt+ \sum_{j=1}^d B_j Xdw^j(t), \hspace{1cm}  0<t\leq T, \\
X(0)=\xi
\end{cases}
\end{equation}
in an M-type 2  separable Banach space $E$.  Here, $A$ is as above, 
 $F$ is a progressively measurable, square integrable process on $[0,T]$, 
 $B_j$ $  (j=1,\dots,d)$ are unbounded linear operators in  $E$, 
and $w^j $ $(j=1,\dots,d)$ are  independent standard Brownian motions.
The author assumed that  
  $$\sum_{j=1}^d\|B_j u\|_{\mathcal D_A(\frac{1}{2},2)}^2 \leq C_1 \|u\|_{\mathcal D(A)}^2 + C_2\|u\|_{\mathcal D_A(\frac{1}{2},2)}^2, \hspace{1cm} u\in \mathcal D(A),$$ 
where $\mathcal D_A(\frac{1}{2},2)=\{u\in E; \int_0^\infty \|Ae^{tA}u\|^2 dt<\infty\}, $ and that  $$\xi \in L^2(\Omega, \mathcal D_A(\frac{1}{2},2)).$$ 
Then,   existence of a unique strict solution to \eqref{E0.5} has been proved in the space 
$$L^2([0,T]\times \Omega;\mathcal D(A)) \cap C([0,T];L^2(\Omega;\mathcal D_A(\frac{1}{2},2))).$$

However, to the best of our knowledge, no one has handled evolution equations of the form 
\eqref{E0} for constructing strict solutions. It is known that many interesting models introduced from the real world can be formulated by deterministic linear or nonlinear evolution equations of parabolic type. These equations generally not only generate a dynamical system but also enjoy a global attractor even a finite-dimensional attractor. Existence of such an attractor then suggests that the phenomena whose processes are described by these parabolic evolution equations enjoy some robustness in a certain abstract sense. Some may be the pattern formation and others may be the specific structure creation. In these cases, robustness of final states of process is one of main issues to be concerned. It is therefore quite natural in order to investigate the robustness to consider an advanced version of stochastic parabolic evolution equations. 

In this paper, we want to treat a rather simple case where the parabolic equation is linear and the noise is additive. But we will construct strict solutions possessing very strong regularities. Such regularities are necessary for constructing a stochastic dynamical system generated by \eqref{E0}.

The  paper is organized as follows. Section \ref{section2} is preliminary. 
Section \ref{section3} presents our main results on strict solutions. Section  \ref{section4} gives an example to illustrate our abstract results.

\section{Preliminary} \label{section2}
\subsection{$E$\,{-}\,valued stochastic processes}
Let us first restate  the Kolmogorov continuity theorem. 
For  $\sigma>0$ and $a<b$,  denote by $\mathcal C^\sigma([a,b];E)$  the space of  $E$\,{-}\,valued functions which are H\"{o}lder continuous on $[a,b]$ with exponent $\sigma$. 
The Kolmogorov continuity theorem gives a sufficient condition for a stochastic process to be H\"{o}lder continuous.
\begin{theorem} \label{Kol1}
Let $\zeta$ be  an $E$\,{-}\,valued stochastic process on $[0,T]$. Assume  that  for some  $c>0,$ $ \epsilon_1>1 $ and $ \epsilon_2>0,$ 
\begin{equation} \label{E11}
\mathbb E\|\zeta(t)-\zeta(s)\|^{\epsilon_1}\leq c (t-s)^{1+\epsilon_2}, \hspace{1cm} 0\leq s \leq t \leq T.
\end{equation}
Then, $\zeta$ has  a version whose $\mathbb P$\,{-}\,almost all trajectories are H\"{o}lder continuous functions with an  arbitrarily smaller  exponent than $\frac{\epsilon_2}{\epsilon_1}$.
\end{theorem}
When  $\zeta$ is a Gaussian process, one can weaken the condition \eqref{E11}.
\begin{theorem}  \label{Kol2}
Let $\zeta$ be an $E$\,{-}\,valued   Gaussian process on $[0,T]$ such that $\mathbb E \zeta(t)=0$ for  $ t\geq 0$.  Assume that  for some  $c>0$ and $ 0<\epsilon\leq 1,$
$$
\mathbb E\|\zeta(t)-\zeta(s)\|^2\leq c (t-s)^\epsilon, \hspace{1cm} 0\leq s \leq t \leq T.
$$
Then, there exists a modification of $\zeta$ whose $\mathbb P$\,{-}\,almost all trajectories are H\"{o}lder continuous functions with an  arbitrarily smaller  exponent than $\frac{\epsilon}{2}$.
\end{theorem}
For the proofs of Theorems \ref{Kol1} and \ref{Kol2}, see, e.g.,    \cite{prato}.

Let us next  review the notion of  Banach spaces of M-type $2$ and   some properties of  stochastic integrals. 
\begin{definition}[Pisier \cite{Pisier}]  \label{MType2BanachSpace}
A Banach space $E$ is said to be of martingale type $2$ (or M-type 2 for abbreviation),  if there is a constant $c(E)$ such that for all $E$\,{-}\,valued martingales $\{M_n\}_n,$ it holds true that 
$$\sup_n \mathbb E\|M_n\|^2 \leq c(E) \sum_{n\geq 0}\mathbb E \|M_n-M_{n-1}\|^2,$$
where $M_{-1}=0$.
\end{definition}
It is known that the Hilbert space is of M-type $2$ and that, when $2\leq p<\infty,$
the $L^p$ space is the same.

When $E$ is of M-type 2, stochastic integrals in  $E$ can be constructed in a quite similar way as for the  usual  It\^{o}  integrals.

\begin{definition}   \label{Def1}
The set of all $L(\mathbb C^d;E)$\,{-}\,valued progressively measurable  processes $\phi$ such that  
$$\mathbb E\int_0^T \|\phi(t)\|_{L(\mathbb C^d;E)}^2 dt<\infty$$
(resp.
$$\int_0^T \|\phi(t)\|_{L(\mathbb C^d;E)}^2 dt<\infty \hspace{1cm} \text{a.s.})$$
is denoted by   $\mathcal N^2(0,T)$ (resp. $\ \mathcal N(0,T)$).
\end{definition}
One can construct stochastic  integrals $\int_0^t \phi(s)dw_s$ for all $\phi \in \mathcal N^2(0,T)$. By localization procedure, the class of integrand can be extended  to $\ \mathcal N(0,T),$ too, see   \cite{Dettweiler}.

\begin{theorem}\label{IntegralInequality}
Let $\{w_t, t\geq 0\}$ be  a $\mathbb C^d$\,{-}\,valued  Brownian motion  on a  filtered probability space  $(\Omega, \mathcal F,\mathcal F_t, \mathbb P)$. Let $E$ be a Banach space  of M-type $2$.  Then, 
\begin{itemize}
  \item [{\rm (i)}] There exists  $c(E)>0$ depending only on  $ E$ such that for  $\phi \in \mathcal N^2(0,T),$
  $$\mathbb E\Big\|\int_0^t \phi(s)dw_s\Big\|^2\leq c(E) \int_0^t \mathbb E\|\phi(s)\|_{L(\mathbb C^d;E)}^2ds, \hspace{1cm} 0\leq t\leq T.$$
  \item [{\rm (ii)}] $\{\int_0^t \phi(s)dw_s, 0\leq t\leq T\}$ is an $E$\,{-}\,valued continuous martingale.
  \item [{\rm (iii)}] (Burkholder-Davis-Gundy inequality) For any $p>1,$ there exists  $c_p(E)$ $>0$   depending only on  $p$ and  $ E$ such that for  $\phi \in \mathcal N^2(0,T)$, 
$$\mathbb E \sup_{t\in [0,T]} \Big\|\int_0^t \phi(s) dw_s\Big\|^p \leq \Big(\frac{p}{p-1}\Big)^p c_p(E) \mathbb E\Big[\int_0^T \|\phi(s)\|_{L(\mathbb C^d;E)}^2 ds\Big]^{\frac{p}{2}}.$$
\end{itemize}
\end{theorem}

For the proof, see \cite{Dettweiler}.  

\begin{proposition}   \label{pro1}
Let $\{w_t, t\geq 0\}$ be  a $\mathbb C^d$\,{-}\,valued  Brownian motion  on a  filtered probability space  $(\Omega, \mathcal F,\mathcal F_t, \mathbb P)$. Let $E$ be a Banach space  of M-type $2$.  Let $B$ be a closed linear operator in $E$ and $\phi\colon [0,T] \subset \mathbb R\to (L(\mathbb C^d;\mathcal D(B)).$
 If $\phi$ and $B\phi$  belong to $  \mathcal N^2(0,T), $   then
$$B\int_0^T \phi(t)dw_t=\int_0^T B\phi(t)dw_t  \hspace{1cm}  \text{  a.s.}$$
\end{proposition}
The proof of Proposition \ref{pro1} is very similar to one in   \cite{prato}. So, we omit it.
\subsection{Deterministic linear evolution equations}
Consider the Cauchy problem for a deterministic linear evolution equation
%%%%%%%%% E2
\begin{equation} \label{E1*}
\begin{cases}
\frac{dX}{dt}+A(t)X=F(t),\hspace{1cm} 0<t\leq T,\\
X(0)=\xi
\end{cases}
\end{equation}
in  $E$. Here, $A(t), t\geq 0,$ are densely defined, closed linear operators in $E$ satisfying the following   conditions. 
  \begin{itemize}
  \item [{\rm (A1)}] For  $0\leq t\leq T$, the spectrum $\sigma(A(t))$ and the resolvent of $A(t)$        
 satisfy
       %%%%%%%%%%%%%%    \label{H1}
\begin{equation*} \label{E2} 
\sigma(A(t)) \subset  \Sigma_{\varpi}=\{\lambda \in \mathbb C: |\arg \lambda|<\varpi\} 
       \end{equation*}
 and  
 %%%%%%%%%%%%%% \label{H2}    
\begin{equation*} \label{E3}
          \|(\lambda-A(t))^{-1}\| \leq \frac{M_{\varpi}}{|\lambda|}, \hspace{1cm}   \lambda \notin \Sigma_{\varpi}
     \end{equation*}
     with some   $\varpi\in (0,\frac{\pi}{2})$ and $M_{\varpi}>0$.
  \item  [{\rm (A2)}] There exists $0<\nu\leq 1$ such that
    \begin{equation*}\label{H3}
    \mathcal D(A(s))\subset \mathcal D(A(t)^\nu), \hspace{1cm}  0\leq s,t\leq T.
    \end{equation*}
    \item [{\rm (A3)}] There exist  $1-\nu<\mu \leq 1$ and  $N>0$ such that 
      \begin{equation*}
   \|A(t)^\nu[A(t)^{-1}-A(s)^{-1}]\|\leq N(t-s)^\mu, \hspace{1cm}  0\leq s\leq t\leq T.
   \end{equation*}
   \end{itemize} 
Meanwhile, 
  $F$ is an $E$\,{-}\,valued function on $[0,T]$ belonging to 
\begin{itemize}
  \item [(F)] \hspace{1cm}  $ F\in \mathcal F^{\beta, \sigma} ((0,T];E)$,  where $0<\beta\leq 1$ \\
   \hspace*{1cm}  and  $0<\sigma< \min\{\beta,\mu+\nu-1\}.$ 
 \end{itemize}
Here,  $\mathcal F^{\beta, \sigma}((0,T];E)$ denotes  a weighted H\"{o}lder continuous function space introduced in  \cite{yagi} which consists of all $E$\,{-}\,valued continuous functions $f$ on $(0,T]$ (resp. $[0,T]$) when $0<\beta<1$ (resp. $\beta=1$)  with  the properties:
\begin{itemize}
  \item [\rm (i)] When $\beta<1$, 
$  
  t^{1-\beta} f(t) $  has a limit as $ t\to 0.
  $
  \item [\rm (ii)]  $f$ is H\"{o}lder continuous with   exponent $\sigma$ and with  weight $s^{1-\beta+\sigma}$, i.e., 
  \begin{equation*} 
\begin{aligned}
&\sup_{0\leq s<t\leq T} \frac{s^{1-\beta+\sigma}\|f(t)-f(s)\|}{(t-s)^\sigma}\\
&=\sup_{0\leq t\leq T}\sup_{0\leq s<t}\frac{s^{1-\beta+\sigma}\|f(t)-f(s)\|}{(t-s)^\sigma}<\infty.
\end{aligned}
\end{equation*}
  \item [\rm (iii)] 
  \begin{equation*} 
  \lim_{t\to 0} w_{f}(t)=0,
  \end{equation*}
  where $w_{f}(t)=\sup_{0\leq s  <t}\frac{s^{1-\beta+\sigma}\|f(t)-f(s)\|}{(t-s)^\sigma}$.
\end{itemize}  
It is clear that   $\mathcal F^{\beta, \sigma}((0,T];E)$ is a Banach space with  norm
$$\|f\|_{\mathcal F^{\beta, \sigma}}=\sup_{0\leq t\leq T} t^{1-\beta} \|f(t)\|+ \sup_{0\leq s<t\leq T} \frac{s^{1-\beta+\sigma}\|f(t)-f(s)\|}{(t-s)^\sigma}.$$
 By the definition, for  $ f\in \mathcal F^{\beta, \sigma}((0,T];E),$  
%%%%%%%% \label{H12}
\begin{equation} \label{E12}  
\begin{aligned}
\|f(t)\| & \leq   \|f\|_{\mathcal F^{\beta, \sigma}} t^{\beta-1}, \hspace{1cm} 0<t\leq T,\\
 \|f(t)-f(s)\|  & \leq  w_{f}(t) (t-s)^{\sigma} s^{\beta-\sigma-1}\\
& \leq \|f\|_{\mathcal F^{\beta, \sigma}} (t-s)^{\sigma} s^{\beta-\sigma-1},  \hspace{1cm} 0<s\leq t\leq T.
\end{aligned}
\end{equation}

According to  \cite{yagi0,yagi}, the following results are known.
\begin{theorem}
Let $A(t), 0\leq t\leq T,$ satisfy {\rm (A1)},  {\rm (A2)} and  {\rm (A3)}. Then, there exists a unique evolution operator $U(t,s), 0\leq s\leq t\leq T, $ having the following properties:
\begin{itemize}
\item  [\rm (i)] $U(t,s)$ is a bounded linear operator on $E$ with
\begin{equation*}
\begin{aligned}
\begin{cases}
U(t,r)=U(t,s)U(s,r), &\hspace{1cm} 0\leq r\leq s\leq t\leq T,\\
U(s,s)=I, &\hspace{1cm} 0\leq s\leq T.
\end{cases}
\end{aligned}
\end{equation*}
\item  [\rm (ii)] For  $0\leq \theta< \mu+\nu,$ $\mathcal R(U(t,s))\subset \mathcal D(A(t)^\theta),$ and   there exists $\iota_\theta>0$ such that
%%%%%%%\label{H10}
\begin{equation} \label{E4}
\|A(t)^\theta U(t,s)\| \leq \iota_\theta (t-s)^{-\theta}, \hspace{1cm}  0\leq s< t\leq T.
\end{equation}
   \item  [\rm (iii)]   For  $0\leq \theta_1\leq  1$ and $\theta_1 \leq \theta_2< \mu+\nu,$ there exists  $\kappa_{\theta_2}>0$ such that
\begin{equation} \label{E8}
\|A(t)^{\theta_2}U(t,s)A(s)^{\theta_1-\theta_2}\| \leq \kappa_{\theta_2}  (t-s)^{-\theta_1}, \hspace{1cm} 0\leq s<t\leq T.
\end{equation}
\item  [\rm (iv)] There exists $c_{\mu,\nu}>0$ such that
\begin{align} 
 \|& A(t)U(t,s)A(s)^{-1}-e^{-(t-s)A(s)}\|    \label{E9}  \\
& \leq c_{\mu,\nu} (t-s)^{\mu+\nu-1}, \hspace{1cm}  0\leq s<t\leq T,   \notag 
\end{align}
where $\{e^{-\tau A(s)}\}_{\tau\geq 0}$ is the semigroup generated by $A(s).$
\end{itemize}
In addition, for  $0\leq \theta\leq 1,$
 \begin{equation}  \label{E6}
   \|A(t)^{-\theta}\|\leq \iota_\theta,  \hspace{1cm}  0\leq t\leq T,
   \end{equation}
where $\iota_\theta$ is the constant in \eqref{E4}.
\end{theorem}
\begin{theorem} \label{maximal regularity}
Let {\rm (A1)}, {\rm (A2)}, {\rm (A3)} and  {\rm (F)} be satisfied.  
Let 
$\xi $ be any value in $ \mathcal D(A(0)^\beta).$ Then, there exists a unique  solution $X$ to  \eqref{E1*} in the function space: 
\begin{equation*}
  X \in \mathcal C^1((0,T]; E), \quad  AX \in \mathcal C((0,T]; E), \quad A^\beta X \in C([0,T]; E).
\end{equation*}
Furthermore, the solution $X$ has the regularity
\begin{equation*}
  \frac{dX}{dt}, AX \in \mathcal F^{\beta,\sigma}((0,T]; E)
\end{equation*}
together with the estimate
$$\Big|\Big|\frac{dX}{dt}\Big|\Big|_{\mathcal F^{\beta,\sigma}}+\|A^\beta X\|_{\mathcal C}+\|AX\|_{\mathcal F^{\beta,\sigma}}\leq C[\|A(0)^\beta \xi\|+\|F\|_{\mathcal F^{\beta,\sigma}}],$$
where $C>0$ is some constant.  
\end{theorem}

As a matter of fact, the solution $X$ is given by the formula
\begin{equation*}
X(t)=U(t,0) \xi+\int_0^t U(t,s)F(s)ds,\hspace{1cm} 0\leq t\leq T.
\end{equation*}

\section{Main results}   \label{section3}
Let us restate the problem we are considering in this paper. 
We  consider the Cauchy problem for a     stochastic evolution equation
%%%%%%%%% E2
\begin{equation} \label{E1}
\begin{cases}
dX+A(t)Xdt=F(t)dt+ G(t)dw_t,\hspace{1cm} 0<t\leq T,\\
X(0)=\xi
\end{cases}
\end{equation}
in a  complex separable M-type 2 Banach space $E$  with  norm $\|\cdot\|$ and the Borel $\sigma$\,{-}\,field  $\mathcal B(E)$. 
Here, 
\begin{itemize}
  \item [\rm{(i)}] $A(t), 0\leq t\leq T,$ are densely defined, closed linear operators  in $E$ satisfying {\rm (A1)}, {\rm (A2)} and {\rm (A3)}.
  \item [\rm{(ii)}] $\{w_t, t\geq 0\}$ denotes  a $\mathbb C^d$\,{-}\,valued  Brownian motion  on a  complete probability space  $(\Omega, \mathcal F,\mathbb P)$ with  filtration $\{\mathcal F_t\}_{t\geq 0}$ satisfying the usual conditions.
  \item  [\rm{(iii)}] $F$ is an $E$\,{-}\,valued progressively measurable  process  on  $[0,T]$ and satisfies {\rm (F)} a.s.  
  \item  [\rm{(iv)}] $G\in \mathcal N^2(0,T)$, where $\mathcal N^2(0,T)$ is defined by Definition \ref{Def1}.
  \item  [\rm{(v)}] $ \xi$ is an $E$\,{-}\,valued $\mathcal  F_0$\,{-}\,measurable  random variable.
  \item  [\rm{(vi)}] $X$ is an unknown $E$\,{-}\,valued process on $[0,T]$. 
\end{itemize}

Let us introduce a definition of strict solutions to   \eqref{E1}. 
%%%%%%% Def2
\begin{definition} \label{Def2}
An $E$\,{-}\,valued  adapted continuous  process $X$ on $ [0,T]$ is called a  strict solution  of \eqref{E1} if 
  $$X(t) \in \mathcal D(A(t))\hspace{1cm}  \text{  a.s.,} \quad  0< t\leq T,$$% \quad \text{ and  } 
$$\int_0^T \|A(s)X(s)\| ds<\infty \hspace{1cm}  \text{  a.s.,}$$
and 
$$
X(t)=\xi -\int_0^t A(s)X(s)ds+\int_0^t F(s)ds+ \int_0^t G(s)dw_s \hspace{0.6cm} \text{a.s.,} \quad 0< t\leq T.
$$
A strict solution $X$ on $[0,T]$ is said to be unique if any other strict solution $\bar X$ on $[0,T]$ is indistinguishable from it, which means that 
$$\mathbb P\{X(t)=\bar X(t) \text{ for every } 0\leq t\leq T\}=1.$$
\end{definition}

\subsection{Uniqueness of strict  solutions}
In this subsection, we  prove uniqueness of strict  solutions to  \eqref{E1}.

\begin{theorem}   \label{uniqueness}
If there exists a strict solution to  \eqref{E1}, then it is unique. 
\end{theorem}

\begin{proof}
Let $X$ and $\bar X$  be two strict solutions of \eqref{E1}. Put $Y(t)=X(t)-\bar X(t)$. From the definition of strict solutions, we have
$$
\begin{cases}
Y(t)=-\int_0^t A(s)Y(s)ds \hspace{1cm} \text{a.s.,} \quad 0<t\leq T,\\
Y(0)=0.
\end{cases}
$$

For each $n=1,2,3,\dots,$ let  $U_n(t,s), 0\leq s\leq t\leq T$, be the evolution operator for the family of Yosida approximations $A_n(t)$ of $A(t)$'s. It is known that (see \cite{yagi})
$$\frac{\partial U_n(t,s)}{\partial s} =U_n(t,s) A_n(s).$$
Then,
\begin{align*}
 \frac{\partial }{\partial r}U_n(t,r)Y(r)&=U_n(t,r) A_n(r)Y(r)-U_n(t,r)A(r)Y(r) \\
& =U_n(t,r)[A_n(r)A(r)^{-1}-I]  A(r)Y(r)  \hspace{1cm} \text{ a.s.}
\end{align*}
Let $0<\epsilon\leq T$.  Integrating this on $[\epsilon,T]$ yields that 
\begin{align*}
Y(t)-U_n(t,\epsilon) Y(\epsilon)=\int_\epsilon^t &U_n(t,r)[A_n(r)A(r)^{-1}-I]   A(r)Y(r) dr \\
& \text{ a.s.,  }  0<\epsilon<t\leq T.
\end{align*}
 Letting $n\to \infty$, it follows that 
$$Y(t)=U(t,\epsilon) Y(\epsilon) \hspace{1cm} \text{ a.s.,  }    0<\epsilon<t\leq T.$$
Letting $\epsilon\to 0$, we arrive at
 $$Y(t)=0     \hspace{1cm} \text{ a.s.},   0< t\leq T.$$
Thus, $X\equiv \bar X $ a.s. on $[0,T].$  By the continuity of $X$ and $ \bar X $ on $[0,T],$
they are indistinguishable.
\end{proof}

\subsection{Existence of strict solutions} 
In this subsection, we  construct strict solutions to \eqref{E1} based on solution formula. 
We  assume that the process $G$   satisfies the condition:
 \begin{itemize}
    \item  [(G)] There exist  $\delta >\frac{1}{2}  $  and a square integrable random variable $\zeta$ such that 
    $$G(t) \in \mathcal D(A(t)^\delta)  \hspace{1cm}  \text{ a.s.,}\quad 0\leq  t\leq T$$
and
   $$\|A(t)^\delta G(t)-A(s)^\delta G(s)\|_{L(\mathbb C^d;E)} \leq \zeta (t-s)^\sigma  \hspace{0.5cm}  \text{ a.s.,     } 0\leq s\leq t\leq T.$$
In addition,  $ \mathbb E  \|A(0)^\delta G(0)\|_{L(\mathbb C^d;E)}^2<\infty.$
\end{itemize}
\begin{theorem} \label{strictSolutions}
%Let $E$ be of M-type $2$.
 Let {\rm (A1)}, {\rm (A2)}, {\rm (A3)}, {\rm (F)}  and
 {\rm (G)} be satisfied. Assume that $\xi \in \mathcal D(A(0)^\beta) $ a.s.  and  $\mathbb E \|A(0)^\beta \xi\|^2<\infty$. 
 Then, there exists a unique strict solution of \eqref{E1} possessing the regularities
$$ A^{\beta}X\in \mathcal C([0,T];E), \quad X\in  \mathcal C^{\gamma_1}([0,T];E) \hspace{1cm} \text{ a.s.,}$$
and 
$$ AX\in \mathcal C^{\gamma_2}([\epsilon,T];E) \hspace{1cm} \text{ a.s.}$$
 for any $  0<\gamma_1\leq \min\{\beta,\frac{1}{2}\},$  $\gamma_1\ne \frac{1}{2},$ $ 0<\gamma_2<\min\{\delta-\frac{1}{2},\sigma\}$ and $0< \epsilon \leq T$. In addition, $X$ satisfies the estimate
\begin{align}  
\mathbb E & \|A(t)^\beta X(t)\|^2   \label{E13}\\ 
\leq &C[\mathbb E\|A(0)^\beta \xi\|^2+\mathbb E\|F\|_{\mathcal F^{\beta,\sigma}}^2   +\mathbb E\|A(0)^\delta G(0)\|_{L(\mathbb C^d;E)}^2   
t^{1-2(\beta-\delta)}   \notag\\
&  +t^{1-2(\beta-\delta)+2\sigma}], \hspace{1cm} 0\leq t\leq T \notag
\end{align}
when $\beta\geq \delta,$ 
and the estimate
\begin{align} 
\mathbb E  \|A(t)^\beta X(t)\|^2  % \\
\leq & C[\mathbb E\|A(0)^\beta \xi\|^2+\mathbb E\|F\|_{\mathcal F^{\beta,\sigma}}^2        
       \label{E13.1}\\
&+\mathbb E\|A(0)^\delta G(0)\|_{L(\mathbb C^d;E)}^2 t+t^{1+2\sigma}], \hspace{1cm} 0\leq t\leq T \notag
\end{align}
when $\beta< \delta,$  where $C>0$ is some constant depending on the exponents. 
Furthermore, 
if $G(0)=0$, then  for any  $ 0<\gamma_1\leq \min\{\frac{1+\sigma}{2}, \delta,$ $\beta\}, $  $\gamma_1\notin \{ \frac{1+\sigma}{2},\delta\},$  
  $$X\in \mathcal C^{\gamma_1}([0,T];E) \hspace{1cm} \text{ a.s.}$$
\end{theorem}

\begin{proof}
Uniqueness of strict solutions has already  been verified by Theorem 
\ref{uniqueness}. It suffices to construct a strict solution based on the solution formula.  We divide the proof into several steps. Throughout the proof, we denote by $C$ some universal constant, which depends on the exponents.

{\bf Step 1}. For  $0\leq \theta\leq 1,$ put
$$W_\theta(t)=\int_0^t A(t)^\theta U(t,s)G(s)dw_s.$$
Let us verify that   $W_\theta$ is well-defined on $[0,T]$ and satisfies
$$W_\theta(t)=A(t)^\theta \int_0^t U(t,s)G(s)dw_s.$$

From {\rm (G)}, it follows that  
\begin{equation}  \label{E14}
\|A(t)^\delta G(t)\|_{L(\mathbb C^d;E)}\leq \|A(0)^\delta G(0)\|_{L(\mathbb C^d;E)} +\zeta  t^\sigma, \hspace{1cm} 0\leq t\leq T. 
\end{equation}
Therefore,
\begin{align*}
&\int_0^t \mathbb E \|A(t)^\theta U(t,s)G(s)\|_{L(\mathbb C^d;E)}^2ds\\
&=\int_0^t \mathbb E \|A(t)^\theta U(t,s)A(s)^{-\delta} A(s)^\delta G(s)\|_{L(\mathbb C^d;E)}^2ds\\
&\leq \int_0^t \|A(t)^\theta U(t,s)A(s)^{-\delta}\|^2  \mathbb E\|A(s)^\delta G(s)\|_{L(\mathbb C^d;E)}^2ds\\
&\leq 2 \int_0^t \|A(t)^\theta U(t,s)A(s)^{-\delta}\|^2 [ \mathbb E\|A(0)^\delta G(0)\|_{L(\mathbb C^d;E)}^2 + \mathbb E\zeta^2  s^{2\sigma}] ds.
\end{align*}
If $\theta\geq \delta$, then  \eqref{E8} gives
\begin{align}
\int_0^t & \mathbb E \|A(t)^\theta U(t,s)G(s)\|_{L(\mathbb C^d;E)}^2ds \notag\\
\leq &2 \int_0^t \|A(t)^\theta U(t,s)A(s)^{\theta-\delta-\theta}\|^2 [ \mathbb E\|A(0)^\delta G(0)\|_{L(\mathbb C^d;E)}^2 + \mathbb E\zeta^2  s^{2\sigma}] ds \notag\\
%3
\leq & 2 \kappa_\theta^2\int_0^t (t-s)^{-2(\theta-\delta)}  [ \mathbb E\|A(0)^\delta G(0)\|_{L(\mathbb C^d;E)}^2  + \mathbb E\zeta^2  s^{2\sigma}]  ds \notag\\
= &\frac{2\kappa_\theta^2 \mathbb E\|A(0)^\delta G(0)\|_{L(\mathbb C^d;E)}^2 t^{1-2(\theta-\delta)}}{1-2(\theta-\delta)} \label{E14.1} \\
&+2 \mathbb E\zeta^2 \kappa_\theta^2  B(1+2\sigma,1-2(\theta-\delta))  t^{1-2(\theta-\delta)+2\sigma}  \notag \\
\leq &C,  \hspace{1cm} 0\leq t\leq T, \notag
\end{align}
where $ B(\cdot,\cdot)$ is the beta function. Meanwhile, if $\theta< \delta$, then  \eqref{E8} and  \eqref{E6} give
\begin{align}
&\int_0^t  \mathbb E\|A(t)^\theta U(t,s)G(s)\|_{L(\mathbb C^d;E)}^2ds \notag\\
&\leq  2 \int_0^t \|A(t)^\theta U(t,s)A(s)^{-\theta}\|^2   \notag\\
& \hspace{1cm} \times \|A(s)^{\theta-\delta}\|^2 [ \mathbb E \|A(0)^\delta G(0)\|_{L(\mathbb C^d;E)}^2 + \mathbb E \zeta^2  s^{2\sigma}]  ds \notag\\
%3
&\leq  2 \kappa_\theta^2 \iota_{\delta-\theta}^2  \int_0^t   [ \mathbb E \|A(0)^\delta G(0)\|_{L(\mathbb C^d;E)}^2 + \mathbb E \zeta^2  s^{2\sigma}] ds \notag\\
&= 2\kappa_\theta^2 \iota_{\delta-\theta}^2  \Big [\mathbb E\|A(0)^\delta G(0)\|_{L(\mathbb C^d;E)}^2 t+\frac{\mathbb E\zeta^2  t^{1+2\sigma}}{1+2\sigma }\Big]   \label{E14.2}\\
&\leq C, \hspace{1cm}    0\leq t\leq T.    \notag
\end{align}
Hence,  $W_\theta$ is well-defined  on $[0,T].$ Since $A(t)^\theta$ is closed, Proposition  \ref{pro1} yields that
$$W_\theta(t)=A(t)^\theta \int_0^t U(t,s)G(s)dw_s, \hspace{1cm}  0\leq t\leq T.$$

{\bf Step 2}. Fix $ 0<\gamma_2<\min\{\delta-\frac{1}{2},\sigma\}.$ 
  For $\theta=1$, let us verify that 
$$ W_1\in  \mathcal C^{\gamma_2}([0,T];E) \hspace{1cm} \text{ a.s.}$$
by using Theorem \ref{Kol2}.

 Let $0<\gamma_2<\gamma_3<\min\{\delta-\frac{1}{2},\sigma\}.$ 
We  prove that
   \begin{align}
\mathbb E\|W_1(t)-W_1(s)\|^2 \leq & C (t-s)^{2\gamma_3}, \hspace{1cm} 0\leq s\leq t\leq T.  \label{E14.3}
\end{align}
 Indeed, noting the expression 
\begin{align*}
W_1(t)=&\int_0^t A(t)U(t,r)A(r)^{-\delta} A(r)^\delta G(r)dw_r,
\end{align*}
we have
\begin{align*}
&W_1(t)-W_1(s)\\
=&\int_s^t A(t)U(t,r)A(r)^{-\delta}A(r)^\delta G(r)dw_r\\
&+\int_0^sA(t)U(t,r)A(r)^{-\delta} A(r)^\delta G(r)dw_r\\
&-\int_0^s A(s)U(s,r)A(r)^{-\delta}A(r)^\delta G(r)dw_r\\
%2
=&\int_s^t A(t)U(t,r)A(r)^{-\delta}[A(r)^\delta G(r)-A(t)^\delta G(t)+A(t)^\delta G(t)]dw_r\\
&+\int_0^s [A(t)U(t,s)-A(s)]U(s,r)A(r)^{-\delta} A(r)^\delta G(r)dw_r\\
%3
=&\int_s^t A(t)U(t,r)A(r)^{-\delta}[A(r)^\delta G(r)-A(t)^\delta G(t)]dw_r\\
&+\int_s^t A(t)U(t,r)A(r)^{-\delta}A(t)^\delta G(t)dw_r\\
&+\int_0^s [A(t)U(t,s)-A(s)]U(s,r)A(r)^{-\delta} A(r)^\delta G(r)dw_r\\
=&J_1+J_2+J_3, \hspace{1cm} 0\leq s \leq t \leq T.
\end{align*}
So,
$$\mathbb E\|W_1(t)-W_1(s)\|^2 \leq 3\sum_{i=1}^3 \mathbb E\|J_i\|^2, \hspace{1cm} 0\leq s \leq t \leq T.$$

The terms $\mathbb E\|J_i\|^2$ $ (i=1,2,3)$ are estimated as follows. By using \eqref{E4} and  {\rm (G)},  we have
\begin{align*}
%1
& \mathbb E\|J_1\|^2  \\
&\leq c(E)\int_s^t \mathbb E \|A(t)U(t,r)A(r)^{-\delta}[A(r)^\delta G(r)-A(t)^\delta G(t)]\|_{L(\mathbb C^d;E)}^2dr\\
%2
&\leq c(E)\int_s^t \|A(t)U(t,r)A(r)^{-\delta}\|^2 \mathbb E\|A(r)^\delta G(r)-A(t)^\delta G(t)\|_{L(\mathbb C^d;E)}^2dr\\
%3
&\leq c(E)\kappa_1^2 \mathbb E\zeta^2  \int_s^t (t-r)^{2(\delta+\sigma-1)} dr\\
%4
&\leq \frac{c(E)\kappa_1^2 \mathbb E \zeta^2 (t-s)^{2(\delta+\sigma)-1}}{2(\delta+\sigma)-1}, \hspace{1cm} 0\leq s \leq t \leq T.  \notag
\end{align*}

Similarly, we have 
\begin{align*}
%1
&\mathbb E\|J_2\|^2 \\
& \leq  c(E)\int_s^t \mathbb E\|A(t)U(t,r)A(r)^{-\delta}A(t)^\delta G(t)\|_{L(\mathbb C^d;E)}^2dr\\
&\leq c(E) \int_s^t \|A(t)U(t,r)A(r)^{-\delta}\|^2 \mathbb E\|A(t)^\delta G(t)\|_{L(\mathbb C^d;E)}^2 dr\\
%2
 & \leq  2c(E) \kappa_1^2   \int_s^t (t-r)^{2(\delta-1)} [\mathbb E\|A(0)^\delta G(0)\|_{L(\mathbb C^d;E)}^2+\mathbb E\zeta^2  t^{2\sigma}]dr \\
%3
&= \frac{2c(E)\kappa_1^2 [\mathbb E\|A(0)^\delta G(0)\|_{L(\mathbb C^d;E)}^2+\mathbb E\zeta^2  t^{2\sigma}] (t-s)^{2\delta-1}}{2\delta-1}, \hspace{0.8cm} 0\leq s \leq t \leq T.
\end{align*}
And,
\begin{align*}
%1
\mathbb E& \|J_3\|^2 \\
\leq  & c(E)\int_0^s\mathbb E \|[A(t)U(t,s)-A(s)]U(s,r)A(r)^{-\delta} A(r)^\delta G(r)\|_{L(\mathbb C^d;E)}^2dr\\
\leq &c(E)\int_0^s \|[A(t)U(t,s)A(s)^{-1}-I]A(s)U(s,r) A(r)^{-\delta}\|^2\\
& \hspace{1.5cm} \times \mathbb E \|A(r)^\delta G(r)\|_{L(\mathbb C^d;E)}^2dr\\
\leq & c(E)\int_0^s \|[A(t)U(t,s)A(s)^{-1}-e^{-(t-s)A(s)}]A(s)U(s,r)A(r)^{-\delta}\|^2\\
& \hspace{1cm} \times \mathbb E [\|A(0)^\delta G(0)\|_{L(\mathbb C^d;E)} +\zeta  r^\sigma]^2dr\\
&+c(E)\int_0^s \|[e^{-(t-s)A(s)}-I]A(s)^{-\theta}\|^2 \|A(s)^{1+\theta}  U(s,r)A(r)^{-\delta}\|^2\\
& \hspace{1cm} \times\mathbb E [\|A(0)^\delta G(0)\|_{L(\mathbb C^d;E)} +\zeta  r^\sigma]^2 dr,  \hspace{1cm} 0\leq s \leq t \leq T.
\end{align*}
Then, \eqref{E8} and \eqref{E9} imply that for any  $0<\theta <\min\{\mu+\nu-1, \delta-\frac{1}{2}\},$ 
\begin{align*}
%1
\mathbb E& \|J_3\|^2 \\
%1
\leq & c(E) c_{\mu,\nu}^2\kappa_1^2 \int_0^s (t-s)^{2(\mu+\nu-1)} (s-r)^{2(\delta-1)}   \\
& \hspace{1.5cm} \times  \mathbb E [\|A(0)^\delta G(0)\|_{L(\mathbb C^d;E)} +\zeta  r^\sigma]^2 dr\\
&+c(E) \iota_\theta^2 \kappa_{1+\theta} \int_0^s  (t-s)^{2\theta}(s-r)^{2(\delta-\theta-1)}  \\
& \hspace{1.5cm} \times \mathbb E[\|A(0)^\delta G(0)\|_{L(\mathbb C^d;E)} +\zeta  r^\sigma]^2  dr\\
%1
\leq & 2c(E) c_{\mu,\nu}^2\kappa_1^2 \int_0^s (t-s)^{2(\mu+\nu-1)} (s-r)^{2(\delta-1)}\\
&\hspace{1cm} \times [\mathbb E\|A(0)^\delta G(0)\|_{L(\mathbb C^d;E)}^2+\mathbb E\zeta^2  r^{2\sigma}] dr\\
&+2c(E) \iota_\theta^2 \kappa_{1+\theta} \int_0^s  (t-s)^{2\theta}(s-r)^{2(\delta-\theta-1)}\\
&\hspace{1cm} \times  [\mathbb E\|A(0)^\delta G(0)\|_{L(\mathbb C^d;E)}^2+\mathbb E\zeta^2  r^{2\sigma}]   dr\\
%2
= & 2c(E) c_{\mu,\nu}^2\kappa_1^2 \Big[\frac{\mathbb E\|A(0)^\delta G(0)\|_{L(\mathbb C^d;E)}^2 s^{2\delta-1}}{2\delta-1} \\
& \hspace{2cm}  +\mathbb E\zeta^2 B(1+2\sigma,2\delta-1) s^{2(\delta+\sigma)-1} \Big](t-s)^{2(\mu+\nu-1)}\\
&+2c(E) \iota_\theta^2 \kappa_{1+\theta}  \Big[\frac{\mathbb E\|A(0)^\delta G(0)\|_{L(\mathbb C^d;E)}^2s^{2(\delta-\theta)-1} }{2(\delta-\theta)-1} \\
& + \mathbb E \zeta^2 B(1+2\sigma,2(\delta-\theta)-1)s^{2(\delta-\theta+\sigma)-1}\Big](t-s)^{2\theta},  \hspace{1cm} 0\leq s\leq t\leq T.
\end{align*}
Since $\mu+\nu-1 \geq \sigma$, there exists $c_1>0$ such that 
$$\mathbb E\|J_3\|^2 \leq c_1 (t-s)^{2\gamma_3}, \hspace{1cm} 0\leq s\leq t\leq T.$$

In this way, we conclude that 
\begin{align*}
%1
\mathbb E\|W_1(t)-W_1(s)\|^2 %\leq & 3\sum_{i=1}^3 \mathbb E\|J_i\|^2
%2
\leq C (t-s)^{2\gamma_3},  \hspace{1cm} 0\leq s\leq t\leq T. % \label{Kol}
\end{align*}

On the other hand, thanks to  Step 1, it is possible to see that 
 $  W_1$ is a Gaussian process on $[0,T]$. By  \eqref{E14.3}, 
 Theorem \ref{Kol2} applied to  $  W_1$ yields that   
$$W_1\in \mathcal C^{\gamma_2}([0,T];E) \hspace{1cm} \text{ a.s.}$$

{\bf Step 3}.  Put 
$$
I_1(t)=U(t,0)\xi +\int_0^tU(t,s) F(s)ds,  \hspace{1cm} 0\leq t\leq T,
$$
and 
$$X(t)=I_1(t)+W_0(t),  \hspace{1cm} 0\leq t\leq T.$$
Let us verify that $X$  is a strict solution of \eqref{E1}  possessing the regularity
 $$AX\in \mathcal C((0,T];E) \hspace{1cm} \text{a.s.}$$

 Due to Theorem \ref{maximal regularity},   
 $$
AI_1\in  \mathcal F^{\beta,\sigma}((0,T];E) \hspace{1cm} \text{a.s.}$$
In particular, 
$$
AI_1\in  \mathcal C((0,T];E) \hspace{1cm} \text{a.s.}$$
In addition, 
by \eqref{E12}, 
\begin{align*}
\int_0^t \|A(s)I_1(s)\|ds &\leq \int_0^t \|AI_1\|_{\mathcal F^{\beta,\sigma}} s^{\beta-1}ds\\
&=\frac{\|AI_1\|_{\mathcal F^{\beta,\sigma}} t^\beta}{\beta} <\infty, \hspace{1cm}  0< t\leq T.
\end{align*}
Since $\frac{dI_1}{dt}=-A(t)I_1(t)+F(t),$ we have 
$$
I_1(t)=\xi-\int_0^t A(s)I_1(s)ds +\int_0^t F(s)ds, \hspace{1cm}   0< t\leq T.$$

On the other hand,  by   Step 2,  
$$AW_0=W_1\in \mathcal C([0,T];E)\hspace{1cm} \text{ a.s.}  $$
Then, it suffices to prove that 
%%%%%% \label{I2equation}
\begin{equation} \label{E15}
W_0(t)+ \int_0^t W_1(s)ds=\int_0^t G(s)dw_s, \hspace{1cm}  0\leq  t\leq T.
\end{equation}

To prove    \eqref{E15}, we want to use the Fubini theorem. Before that, we however observe from   \eqref{E6} and \eqref{E14},       that 
\begin{align*}
&\int_0^t \mathbb E\|G(s)\|_{L(\mathbb C^d;E)}^2ds\\
&\leq \int_0^t \|A(s)^{-\delta}\|^2\mathbb E \|A(s)^\delta G(s)\|_{L(\mathbb C^d;E)}^2ds\\
&\leq 2 \int_0^t \iota_\delta^2[\mathbb E \|A(0)^\delta G(0)\|_{L(\mathbb C^d;E)}^2 +\mathbb E\zeta^2  s^{2\sigma}]ds\\
&=2\iota_\delta^2 \mathbb E \|A(0)^\delta G(0)\|_{L(\mathbb C^d;E)}^2 t
+  \frac{2 \mathbb E\zeta^2  \iota_\delta^2   t^{1+2\sigma}}{1+2\sigma}\leq C, \hspace{1cm} 0\leq  t\leq T.
\end{align*}
 Thereby, the stochastic integral $\int_0^t G(s)dw_s$ is well-defined  for $ 0\leq t\leq T.$

 On the other hand,  with $\theta=1$ \eqref{E14.1} gives
\begin{align}
&\int_0^t\mathbb E\|A(t)U(t,s)G(s)\|_{L(\mathbb C^d;E)}^2ds  \label{E16}\\
\leq &\frac{ 2  \kappa_1^2\mathbb E \|A(0)^\delta G(0)\|_{L(\mathbb C^d;E)}^2 t^{2\delta-1} }{2\delta-1}
+  2\mathbb E\zeta^2 \kappa_1^2 B(1+2\sigma,2\delta-1)t^{2(\delta+\sigma) -1}  \notag\\
\leq &C, \hspace{1cm} 0\leq  t\leq T.\notag
\end{align}
 Theorem  \ref{IntegralInequality} and \eqref{E16} then yield that 
\begin{align*}
\mathbb E \Big(\int_0^t \|W_1(s)\|ds\Big)^2 \leq &t \int_0^t \mathbb E\|W_1(s)\|^2ds\\
= &t \int_0^t \mathbb E\Big|\Big|\int_0^s A(s)U(s,u)G(u)dw_u\Big|\Big|^2ds\\
\leq & c(E) t \int_0^t \int_0^s \mathbb E \|A(s)U(s,u)G(u)\|_{L(\mathbb C^d;E)}^2duds\\
\leq &
\frac{ 2c(E)   \kappa_1^2 \mathbb E\|A(0)^\delta G(0)\|_{L(\mathbb C^d;E)}^2 t \int_0^t s^{2\delta-1} ds}{2\delta-1} \\
&+ 2c(E)  \mathbb E \zeta^2 \kappa_1^2 B(1+2\sigma,2\delta-1)t \int_0^t s^{2(\delta+\sigma) -1}ds\\
= &
\frac{ 2c(E)   \kappa_1^2 \mathbb E\|A(0)^\delta G(0)\|_{L(\mathbb C^d;E)}^2 t^{2\delta+1} }{2\delta(2\delta-1)} \\
&+  \frac{2c(E) \mathbb E\zeta^2 \kappa_1^2 B(1+2\sigma,2\delta-1)t^{2(\delta+\sigma) +1}}{2(\delta+\sigma) } \\
\leq &C, \hspace{1cm}  0\leq  t\leq T.
\end{align*}
Hence, for  $0\leq  t\leq T$, the integral  $\int_0^t\|W_1(s)\|ds$ is well-defined a.s. 

It is now  ready to use  the Fubini theorem. We have  
\begin{align*}
 \int_0^t W_1(s)ds&=\int_0^t \int_0^s A(s)U(s,u) G(u)dw_uds\\
&=\int_0^t \int_u^t A(s)U(s,u) G(u)dsdw_u\\
&=\int_0^t [G(u)-U(t,u)G(u)]dw_u\\
&=\int_0^t G(u)dw_u-\int_0^t U(t,u)G(u)dw_u\\
&=\int_0^t G(u)dw_u-W_0(t),  \hspace{1cm}  0\leq  t\leq T.
\end{align*}
Thus,  \eqref{E15}  has been verified.

{\bf Step 4}. Let us verify  that  $A^\beta X\in \mathcal C([0,T];E)$ a.s., and that $X$ satisfies  \eqref{E13} and \eqref{E13.1}.

 We see that 
$$A(t)^\beta X(t)=A(t)^{\beta-1} A(t) X(t), \hspace{1cm} 0<t\leq T.$$
Since $A(\cdot)^{\beta-1}\in \mathcal C([0,T];L(E)),$ the continuity of $A^\beta X$ on $(0,T]$  follows  Step 3. 
In addition, thanks to Theorem \ref{maximal regularity} and Step 2, $A^\beta I_1$ and  $A^\beta W_0$ are  continuous at $t=0$. Therefore, $A^\beta X\in \mathcal C([0,T];E)$  a.s.

On the other hand, due to  \eqref{E14.1} and  \eqref{E14.2}, we observe that 
when
$\beta\geq \delta$,
\begin{align*}
\int_0^t & \mathbb E \|A(t)^\beta U(t,s)G(s)\|_{L(\mathbb C^d;E)}^2ds \notag\\
\leq  &\frac{2\kappa_\beta^2 \mathbb E\|A(0)^\delta G(0)\|_{L(\mathbb C^d;E)}^2 t^{1-2(\beta-\delta)}}{1-2(\beta-\delta)}  \notag\\
&+2 \mathbb E\zeta^2 \kappa_\beta^2  B(1+2\sigma,1-2(\beta-\delta))  t^{1-2(\beta-\delta)+2\sigma}   \\
\leq &C,  \hspace{1cm} 0\leq  t\leq T. \notag
\end{align*}
Meanwhile, when $\beta< \delta$,
\begin{align*}
&\int_0^t  \mathbb E\|A(t)^\beta U(t,s)G(s)\|_{L(\mathbb C^d;E)}^2ds \notag\\
&\leq   2\kappa_\beta^2 \iota_{\delta-\beta}^2  \Big [\mathbb E\|A(0)^\delta G(0)\|_{L(\mathbb C^d;E)}^2 t+\frac{\mathbb E\zeta^2  t^{1+2\sigma}}{1+2\sigma }\Big]  \notag\\
& \leq C, \hspace{1cm} 0\leq  t\leq T. 
\end{align*}
Since
$$\mathbb E \|A(t)^\beta X(t)\|^2 \leq 2 \mathbb E \|A(t)^\beta I_1(t)\|^2 +2 \mathbb E \|A(t)^\beta I_2(t)\|^2$$
and 
$$\mathbb E \|A(t)^\beta I_2(t)\|^2=\int_0^t  \mathbb E\|A(t)^\beta U(t,s)G(s)\|_{L(\mathbb C^d;E)}^2ds,$$
\eqref{E13} and \eqref{E13.1} follow from the above estimates and those of  Theorem \ref{maximal regularity}.

 {\bf Step 5}. Let us show that  for any $ 0<\gamma_1\leq \min\{\beta,\frac{1}{2}\}$ and  $\gamma_1\ne \frac{1}{2}$, 
 $$X\in \mathcal C^{\gamma_1}([0,T];E) \hspace{1cm} \text{ a.s.,}$$
 and that, if $ G(0)=0$, then  for any  $ 0<\gamma_1\leq \min\{\frac{1+\sigma}{2}, \delta, \beta\}, $  $\gamma_1\notin \{ \frac{1+\sigma}{2},\delta\},$  
  $$X\in \mathcal C^{\gamma_1}([0,T];E) \hspace{1cm} \text{ a.s.} $$

First, we  prove  that $I_1$ is $\beta$\,{-}\,H\"{o}lder continuous a.s. on $[0,T]$. According to Theorem \ref{maximal regularity},  $\frac{dI_1}{dt} \in \mathcal F^{\beta,\sigma}((0,T];E)$ a.s.  Then,  \eqref{E12} gives  that
\begin{align*}
\|I_1(t)-I_1(s)\|&=\Big|\Big|\int_s^t \frac{dI_1(u)}{du}du\Big|\Big| \leq \int_s^t  \Big|\Big|\frac{dI_1(u)}{du}\Big|\Big| du\notag\\
&\leq \int_s^t  \Big|\Big|\frac{dI_1}{du}\Big|\Big|_{\mathcal F^{\beta, \sigma}} u^{\beta-1}du 
\leq \Big|\Big|\frac{dI_1}{du}\Big|\Big|_{\mathcal F^{\beta, \sigma}} \frac{t^\beta-s^\beta}{\beta}\notag \\ 
%1 
&\leq  \Big|\Big|\frac{dI_1}{du}\Big|\Big|_{\mathcal F^{\beta, \sigma}}   \frac{(t-s)^\beta}{\beta}, \hspace{1cm} 0\leq s<t\leq T.
\end{align*}
Thus, $I_1\in \mathcal C^\beta([0,T];E)$ a.s.

Next, we  verify that for any  $ 0<\rho<\frac{1}{2},$ 
$$W_0\in \mathcal C^{\rho}([0,T];E) \hspace{1cm} \text{ a.s.} $$
  Let $0\leq s<t\leq T$, then 
\begin{align*}
W_0(t)-W_0(s)=&\int_s^t U(t,r)G(r)dw_r +\int_0^s [U(t,r)-U(s,r)]G(r)dw_r.
\end{align*}
Since the integrals in the right-hand side are independent stochastic functions which  have a zero expectation,  we have
\begin{align*}
&\mathbb E \|W_0(t)-W_0(s)\|^2  \notag\\
%1
=&\mathbb E \Big|\Big|\int_s^t U(t,r)G(r)dw_r\Big|\Big|^2+\mathbb E\Big|\Big|\int_0^s [U(t,r)-U(s,r)]G(r)dw_r\Big|\Big|^2  \notag\\
%2
\leq  &c(E) \int_s^t \mathbb E \|U(t,r)A(r)^{-\delta} A(r)^\delta G(r)\|_{L(\mathbb C^d;E)}^2 dr\notag\\
&+ c(E) \int_0^s \mathbb E \|[U(t,r)-U(s,r)]G(r)\|_{L(\mathbb C^d;E)}^2dr  \notag\\
%4
\leq & c(E) \int_s^t \mathbb E \|U(t,r)\|^2 \|A(r)^{-\delta}\|^2 \|A(r)^\delta G(r)\|_{L(\mathbb C^d;E)}^2 dr\notag\\
&+ c(E)\int_0^s \|[U(t,r)-U(s,r)]A(r)^{-\delta}\|^2 \mathbb E\|A(r)^\delta G(r)\|_{L(\mathbb C^d;E)}^2dr.
\end{align*}
So, by using  \eqref{E4}, \eqref{E6}  and   \eqref{E14}, 
\begin{align*}
\mathbb E& \|W_0(t)-W_0(s)\|^2  \notag\\
%1
\leq &c(E)  \iota_0^2 \iota_\delta^2 \int_s^t \mathbb E[\|A(0)^\delta G(0)\|_{L(\mathbb C^d;E)} +\zeta  r^\sigma]^2dr\notag\\
&+ c(E)\int_0^s\|[U(t,r)-U(s,r)]A(r)^{-\delta}\|^2 \mathbb E[\|A(0)^\delta G(0)\|_{L(\mathbb C^d;E)} +\zeta  r^\sigma]^2dr.  \notag
\end{align*}

The first term in the right-hand side of the latter inequality can be estimated as follows:  
\begin{align*}
&c(E)\iota_0^2 \iota_\delta^2 \int_s^t \mathbb E [\|A(0)^\delta G(0)\|_{L(\mathbb C^d;E)} +\zeta  r^\sigma]^2dr\\
\leq &2 c(E)\iota_0^2 \iota_\delta^2 \int_s^t  [\mathbb E\|A(0)^\delta G(0)\|_{L(\mathbb C^d;E)}^2 +\mathbb E\zeta^2  r^{2\sigma}]dr\\
= &2 c(E)\iota_0^2 \iota_\delta^2\mathbb E \|A(0)^\delta G(0)\|_{L(\mathbb C^d;E)}^2 (t-s) + \frac {2c(E)\mathbb E\zeta^2  \iota_0^2 \iota_\delta^2 (t^{1+2\sigma}-s^{1+2\sigma})}{1+2\sigma}\notag\\
%3
\leq & 2c(E)\iota_0^2 \iota_\delta^2\mathbb E\|A(0)^\delta G(0)\|_{L(\mathbb C^d;E)}^2 (t-s) + \frac {2c(E)\mathbb E\zeta^2  \iota_0^2 \iota_\delta^2 (t-s)^{1+2\sigma}}{1+2\sigma}.\notag
\end{align*}

For the second term, since
\begin{align*}
&\|[U(t,r)-U(s,r)]A(r)^{-\delta}\|^2\\
&=\Big |\Big|\int_s^t A(u)S(u,r) A(r)^{-\delta}du\Big|\Big|^2\\
&\leq  \kappa_1^2 \Big[\int_s^t (u-r)^{\delta-1}du\Big]^2=\frac{\kappa_1^2}{\delta^2}  [(t-r)^\delta-(s-r)^\delta]^2  \\
&\leq  \frac{\kappa_1^2}{\delta^2}  (t-s)^{2\delta},
\end{align*}
we have 
\begin{align*}
&c(E)\int_0^s\|[U(t,r)-U(s,r)]A(r)^{-\delta}\|^2 \mathbb E [\|A(0)^\delta G(0)\|_{L(\mathbb C^d;E)} +\zeta  r^\sigma]^2dr\\
&\leq   \frac{2c(E)\kappa_1^2}{\delta^2}  (t-s)^{2\delta} \int_0^s [\mathbb E\|A(0)^\delta G(0)\|_{L(\mathbb C^d;E)}^2 +\mathbb E\zeta^2  r^{2\sigma}]dr\\
&= \Big[\frac{2c(E)\kappa_1^2 \mathbb E \|A(0)^\delta G(0)\|_{L(\mathbb C^d;E)}^2 s}{\delta^2} 
 + \frac {2c(E)\mathbb E\zeta^2   \kappa_1^2 s^{1+2\sigma}}{(1+2\sigma)\delta^2}\Big]  (t-s)^{2\delta}.\notag
\end{align*}
Thus, 
\begin{align*}
\mathbb E& \|W_0(t)-W_0(s)\|^2  \notag\\
%2
\leq &2 c(E)\iota_0^2 \iota_\delta^2\mathbb E\|A(0)^\delta G(0)\|_{L(\mathbb C^d;E)}^2 (t-s) + \frac {2c(E)\mathbb E\zeta^2  \iota_0^2 \iota_\delta^2 (t-s)^{1+2\sigma}}{1+2\sigma} \notag\\
&+ \Big[\frac{2c(E)\kappa_1^2 \mathbb E \|A(0)^\delta G(0)\|_{L(\mathbb C^d;E)}^2 s}{\delta^2} 
 + \frac {2c(E)\mathbb E\zeta^2   \kappa_1^2 s^{1+2\sigma}}{(1+2\sigma)\delta^2}\Big](t-s)^{2\delta}, \\
&  0\leq s<t\leq T.
\end{align*}

On the other hand, by the definition of  stochastic integrals, $  W_0$ is a Gaussian process on $[0,T]$. We then  apply Theorem \ref{Kol2} to $W_0$ to obtain that, for any $ 0<\rho<\frac{1}{2},$ 
$$W_0\in \mathcal C^{\rho}([0,T];E) \hspace{1cm} \text{ a.s.} $$
In addition,  if $  G(0)=0$, then for any $ 0<\rho<\min\{\frac{1+\sigma}{2},\delta\}, $ 
 $$W_0\in \mathcal C^{\rho}([0,T];E) \hspace{1cm} \text{ a.s.} $$
In this way, we have concluded that  for any  $  0<\gamma_1\leq \min\{\beta,\frac{1}{2}\},$   $\gamma_1\ne \frac{1}{2}$, 
$$X=I_1+W_0\in \mathcal C^{\gamma_1}([0,T];E) \hspace{1cm}  \text{ a.s.}$$
 Furthermore, when $ G(0)=0$,  for any $ 0<\gamma_1\leq \min\{\frac{1+\sigma}{2}, \delta,\beta\}$, $\gamma_1\ne  \{\frac{1+\sigma}{2}, \delta\},$
  $$X\in \mathcal C^{\gamma_1}([0,T];E) \hspace{1cm} \text{ a.s.}$$

{\bf Step 6}. Fix $ 0<\gamma_2<\min\{\delta-\frac{1}{2},\sigma\}$ and $0<\epsilon\leq T$.  Let us prove that  
$$ AX\in  \mathcal C^{\gamma_2}([\epsilon,T];E) \hspace{1cm} \text{ a.s.}$$

  Theorem \ref{maximal regularity}  and Step 2 provide that 
$$AI_1 \in \mathcal F^{\beta,\sigma}((0,T];E)\subset \mathcal C^{\sigma}([\epsilon,T];E)$$
and
$$W_1\in  \mathcal C^{\gamma_2}([0,T]; E) \hspace{1cm} \text{ a.s.}$$
Therefore, $AX=AI_1+ W_1  \in  \mathcal C^{\gamma_2}([\epsilon,T];E) $  a.s.

By  Steps 1-6,   the proof of the theorem is now complete.
\end{proof}
\subsection{Case where $\nu=1$}
In this subsection, we  consider the favorable case where $\nu=1.$ In other words, the domain $\mathcal D(A(t))$ does not dependent on time $t$. 
We show that the  condition  {\rm (G)} can naturally be replaced by the condition: 
 \begin{itemize}
    \item  [(G$'$)] There exist  $ \delta_1 >\frac{1}{2}$    and a square integrable random  variable $  \bar\zeta  $ such that 
    $$G(t) \in \mathcal D(A(0)^{\delta_1}) \hspace{1cm} \text{a.s.}, 0\leq t\leq T,$$
and 
$$  \|A(0)^{\delta_1} [G(t)- G(s)]\| \leq \bar\zeta |t-s|^\sigma  \hspace{0.9cm}  \text{ a.s.,     }  0\leq s\leq t\leq T.$$
In addition,  
$$  \mathbb E  \|A(0)^{\delta_1} G(t)\|_{L(\mathbb C^d;E)}^2<\infty, \hspace{1cm}  0\leq t\leq T. $$
\end{itemize}
\begin{proposition}  \label{lemma4.0}
Under  {\rm (A1)}, {\rm (A2)} and {\rm (A3)} with $\nu=1$, 
if {\rm (G$'$)}  takes place, then so does {\rm (G)}.
\end{proposition}

For the proof, we notice the following lemma.
%%%%%%%%%%%%% Lemma 
\begin{lemma}  \label{lemma4.1}
For any $0<\theta_1<\theta_2<1$, there exists  $\alpha_{\theta_1,\theta_2}>0$ such that 
$$\|[A(t)^{\theta_1} -A(s)^{\theta_1} ] A(s)^{-\theta_2} \| \leq \alpha_{\theta_1,\theta_2} |t-s|^\mu, \hspace{1cm} 0\leq s\leq t\leq T.$$
\end{lemma}
For the proof, see \cite[(3.92)]{yagi} with $\tau=0$.
\begin{proof}[Proof of Proposition \ref{lemma4.0}]
Let $\frac{1}{2}<\delta<\delta'<\delta_1$. For  $0\leq s\leq t\leq T$, we have
\begin{align*}
\|&A(t)^\delta G(t)-A(s)^\delta G(s)\|_{L(\mathbb C^d;E)}\\
%2
=&\|[A(t)^\delta [G(t)-G(s)] +[A(t)^\delta- A(s)^\delta] G(s)\|_{L(\mathbb C^d;E)}\\
%3
\leq & \|[A(t)^\delta-A(0)^\delta+A(0)^\delta] A(0)^{-\delta_1}\|      \| A(0)^{\delta_1} [G(t)-G(s)]\|_{L(\mathbb C^d;E)} \\
& + \|[A(t)^\delta- A(s)^\delta] A(s)^{-\delta'}\| \|[A(s)^{\delta'}-A(0)^{\delta'} + A(0)^{\delta'}] A(0)^{-\delta_1}\|  \\
&\times \|A(0)^{\delta_1} G(s)\|_{L(\mathbb C^d;E)}\\
%4
\leq & \|[A(t)^\delta-A(0)^\delta] A(0)^{-\delta_1}\|      \| A(0)^{\delta_1} [G(t)-G(s)]\|_{L(\mathbb C^d;E)} \\
&+ \|A(0)^{\delta-\delta_1}\|      \| A(0)^{\delta_1} [G(t)-G(s)]\|_{L(\mathbb C^d;E)}\\
& + \|[A(t)^\delta- A(s)^\delta] A(s)^{-\delta'}\| \|[A(s)^{\delta'}-A(0)^{\delta'} ] A(0)^{-\delta_1}\|\\
&\hspace{1cm} \times \|A(0)^{\delta_1} G(s)\|_{L(\mathbb C^d;E)}\\
& + \|[A(t)^\delta- A(s)^\delta] A(s)^{-\delta'}\| \| A(0)^{\delta'-\delta_1}\| \|A(0)^{\delta_1} G(s)\|_{L(\mathbb C^d;E)}.
\end{align*}
By using Lemma \ref{lemma4.1} and {\rm (G$'$)}, it is easily seen that  
\begin{align*}
\|&A(t)^\delta G(t)-A(s)^\delta G(s)\|_{L(\mathbb C^d;E)}\\
%2
\leq & [\alpha_{\delta,\delta_1} t^\mu       (t-s)^\sigma + \|A(0)^{-1}\|^{\delta_1-\delta}]    \bar\zeta  (t-s)^\sigma\\
& + \alpha_{\delta,\delta'} [\alpha_{\delta',\delta_1} s^\mu + \|A(0)^{-1}\|^{\delta_1-\delta'}] \\
&\times \|A(0)^{\delta_1} G(s)\|_{L(\mathbb C^d;E)} (t-s)^\mu, \hspace{1cm} 0\leq s\leq t\leq T.
\end{align*}
\end{proof}

\section{An application}  \label{section4}
Our abstract results can be applicable to many stochastic parabolic equations. We here present an example of stochastic diffusion equation formulated in the $L_p $ space $(2 \le p < \infty).$

Consider  a stochastic diffusion equation
\begin{equation}  \label{E17}
\left\{\begin{aligned}
  &dX(x,t) = \left\{\sum_{i,j=1}^n \frac{\partial}{\partial x_j}
             \left[a_{ij}(x,t)\frac{\partial X}{\partial x_i}\right]
             - b(x,t)X + f(t)\varphi_1(x)\right\}dt  \\
  &\hspace{2cm} + g(t)\varphi_2(x)dw_t \hspace{1.4cm}
             \text{in} \quad \mathcal O \times (0,T), \\
  &X = 0  \hspace{5cm}\text{on}\quad  \partial\mathcal O \times (0,T), \\
  &X(x,0) = X_0(x)  \hspace{3.3cm}\text{in}\quad  \mathcal O
\end{aligned} \right.
\end{equation}
in a bounded domain $\mathcal O\subset \mathbb R^n$ with $\mathcal C^2$ boundary $\partial\mathcal O$ $(n=1,2,\dots)$.  Here, $a_{ij} (\cdot,\cdot), 1\leq i,j\leq n,$ and $b(\cdot,\cdot)$ are real-valued functions in $\overline{\mathcal O}\times [0,T]$; $\varphi_1$ and $\varphi_2$ are complex-valued functions in $\mathcal O$;  $\{w_t,t\geq 0\}$ is a one-dimensional Brownian motion on a complete filtered probability space $(\Omega, \mathcal F,\{\mathcal F_t\}_{t\geq 0}, \mathbb P);$
 $f$ and $ g $ are  complex-valued stochastic processes; and $X_0(x)$ is an initial random function.

We assume the following three conditions.
%\begin{enumerate}
%\item [(1)] 
%\begin{equation}
\begin{align}
&\text{For some } 0 < \mu \le 1, a_{ij} \in \mathcal C^\mu([0,T]; \mathcal C^2(\overline{\mathcal O}))   \label{E18}\\
&\text{and } b \in \mathcal C^\mu([0,T]; L_\infty(\mathcal O)).  \notag\\
%\end{aligned}
%\end{equation}
%\item [(2)] 
%\begin{equation}
%\begin{aligned}
&\text{There exists a constant } a_0 > 0 \text{ such that } \label{E19} \\
%\begin{equation*}
& \hspace{2cm} \sum_{i,j=1}^n a_{ij}(x,t) z_iz_j \ge a_0\|z\|_{\mathbb R^n}^2   \notag\\
%\end{equation*}
 &\text{for } z=(z_1,\ldots,z_n) \in \mathbb R^n \text{ and }
   (x,t) \in \overline{\mathcal O} \times [0,T].   \notag\\
%\end{aligned}
%\end{equation}
%\item [(3)] 
%\begin{equation}
%\begin{aligned}
& \text{There exists a constant } b_0 > 0 \text{ such that }  \label{E20} \\
&\hspace{2cm} b(x,t) \ge b_0,\; (x,t) \in \overline{\mathcal O} \times [0,T].   \notag
\end{align}
%\end{equation}

%\end{enumerate}

Let us formulate \eqref{E17} as a problem of the form \eqref{E1}. We set  the underlying space  $E=L_p(\mathcal O), $ where $p$ is a fixed exponent such that $2 \le p < \infty$.
As stated in Section \ref{section2}, $E$ is an M-type $2$ Banach space.

For each $0\leq t \leq T$, let $A(t)$ be a realization of the differential operator
\begin{equation*}
  -\sum_{i,j=1}^n \frac{\partial}{\partial x_j}
   \left[ a_{ij}(x,t)\frac{\partial}{\partial x_j} \right] + b(x,t)
\end{equation*}
in $L_p(\mathcal O)$ under the Dirichlet boundary conditions on $\partial\mathcal O$. According to \cite[Theorems 2.12 and 2.15]{yagi}, the operators $A(t)$ are  sectorial operators of $E$ of angle $\omega $ less than $\frac\pi2$ having a domain 
$$\mathcal D(A(t)) \equiv H^2_{p,D}(\mathcal O) = \{u \in H_p^2(\mathcal O);\; u_{|\partial\mathcal O}=0\},$$
 here $H_p^2(\mathcal O)$ denotes the space of all complex valued-functions whose partial derivatives in the distribution sense up to the second order belong to $L_p(\mathcal O)$. Since $\mathcal D(A(t))$ is independent of $t$, $\{A(t), 0\leq t\leq T\}$  satisfies {\rm (A2)}  with $\nu=1$. In addition, it is directly verified from \eqref{E18}  (see \cite[(3.68)]{yagi}) that $\{A(t), 0\leq t\leq T\}$ satisfies {\rm (A3)}, too, with $\nu=1$.

We set also stochastic processes $F(t)=f(t)\varphi_1(x)$ and $G(t)=g(t)\varphi_2(x)$, and assume that
%\begin{enumerate}
%\item [(4)]
\begin{align}
 &\varphi_1 \in L_p(\mathcal O), \text{ and for some } 0 < \beta \le 1 \text{ and } 0 < \sigma < \min\{\beta,\mu\},  \label{E21}\\
%\begin{equation*}
 &\hspace{2cm}  f \in \mathcal F^{\beta,\sigma}((0,T]; \mathbb  C)
  \qquad \text{a.s.}   \notag\\
%\end{equation*}
%\item [(5)]
& \varphi_2 \in H^{1+\varepsilon_1}_{p,D}(\mathcal O) \text{ with some } \varepsilon_1 > 0, \text{ and } g \in \mathcal C^\sigma([0,T]; \mathbb  C) \text{ a.s.}    \label{E22}\\
& \text{with a uniform estimate }     \notag\\
%\begin{equation*}
 &\hspace{2cm}  |g(t)-g(s)| \le C_1(t-s)^\sigma,
  \qquad  0 \le s \le t \le T,\; \text{a.s.}     \notag\\
%\end{equation*}  \par
&\text{and with the condition } \mathbb  E |g(0)|^2 < \infty.  \notag \\
&\text{Here, } C_1>0 \text{ and }   \sigma \text{ is } \text{as above}.   \notag
%\end{enumerate}
\end{align}

Then, {\rm (F)} and {\rm (G)} are fulfilled. Indeed, \eqref{E21}  directly implies {\rm (F)} a.s. As for {\rm (G)}, it suffices to verify that \eqref{E22}  implies (G$'$). Then, as remarked below, $\varphi_2 \in H^{1+\varepsilon_1}_{p,D}(\mathcal O)$ implies $\varphi_2 \in \mathcal D(A(0)^{\delta_1})$ for any exponent $\delta_1 < \frac{1+\varepsilon_1}{2}$.  Therefore, we can take a $\delta_1 > \frac12$ so that $G(t) \in \mathcal D(A(0)^{\delta_1}),\, 0 \le t \le T$, a.s. In addition,
\begin{align*}
 & \|A(0)^{\delta_1}[G(t)-G(s)]\|_{L(\mathbb C;E)}  \\
  &\le \|A(0)^{\delta_1}\varphi_2\|_E \, |g(t)-g(s)|   \\
  &\le C_1 \|A(0)^{\delta_1}\varphi_2\|_E \, (t-s)^\sigma,
   \hspace{1cm}  0 \le s \le t \le T,\; \text{a.s.}
\end{align*}
Similarly,
\begin{align*}
  \mathbb E\|A(0)^{\delta_1}G(t)\|^2_{L(\mathbb C;E)}
  &= \|A(0)^{\delta_1}\varphi_2\|^2_E \, \mathbb E|g(t)|^2  \\
   &\le \|A(0)^{\delta_1}\varphi_2\|^2_E \, \mathbb E[C_1t^\sigma +|g(0)|]^2  \\
  &\le 2\|A(0)^{\delta_1}\varphi_2\|^2_E \,
         [(C_1t^\sigma)^2+\mathbb E|g(0)|^2], \qquad  0 \le t \le T.
\end{align*}
Hence, (G$'$) is verified.

We finally take $X_0$ in such a way that
%\begin{enumerate}
%\item [(6)]
\begin{align}
 &\text{if } \beta \ge \frac1{2p}, \text{ then } X_0 \in H^{2\beta+\varepsilon_2}_{p,D}(\mathcal O) \text{ a.s.~with some }\varepsilon_2 >0  \label{E23}\\
&\text{and with the condition } \mathbb E\|X_0\|^2_{H^{2\beta+\varepsilon_2}_{p,D}} < \infty.    \text{ Meanwhile}, 
 \notag\\
&\text{if } \beta < \frac1{2p}, \text{ then } X_0 \in H^{2\beta+\varepsilon_3}_p(\mathcal O) \text{ a.s.~with } \text{some } \varepsilon_3 >0\notag \\
& \text{and with the condition } \mathbb E\|X_0\|^2_{H^{2\beta+\varepsilon_3}_p} < \infty, \beta \text{ being as above.}  \notag
%\end{enumerate}
\end{align}
Then, as remarked below, \eqref{E23}  implies 
$$X_0 \in \mathcal D(A(0)^\beta) \text{ a.s.~and } \mathbb E\|A(0)^\beta X_0\|^2_E < \infty.$$

We have thus verified that all the structural assumptions of Theorem \ref{strictSolutions} are fulfilled for the problem \eqref{E17}.

\begin{proposition}
Let $2 \le p < \infty$. Under \eqref{E18}$\sim$\eqref{E23}, there exists a unique strict solution $X$ to \eqref{E17} in the function space:
\begin{equation*}
  X \in \mathcal C^{\gamma_1}([0,T]; L_p(\mathcal O), \quad
  A^\beta X \in \mathcal C([0,T]; L_p(\mathcal O))  \qquad  a.s.,
\end{equation*}
and
\begin{equation*}
  AX \in \mathcal C^{\gamma_2}([\tau,T]; L_p(\mathcal O))  \qquad  a.s.
\end{equation*}
with any exponents $0 < \gamma_1 < \min\{\beta,\tfrac12\}$ and $0 < \gamma_2 < \min\{\delta-\tfrac12,\sigma\}$ for arbitrarily fixed time $\tau > 0$.
\par

Furthermore, if $\beta \ge \delta$, then
\begin{multline*}
  \mathbb E\|A(t)^\beta X(t)\|^2_{L_p} \le C[\mathbb E\|A(0)^\beta X_0\|^2_{L_p}
      + \mathbb E\|f\|_{\mathcal F^{\beta,\sigma}}^2\|\varphi_1\|^2_{L_p}  \\
      + \|A(0)^\delta\varphi_2\|_{L_p}^2\mathbb E |g(0)|^2t^{1-2(\beta-\delta)}
      + t^{1-2(\beta-\delta)+2\sigma}],
      \qquad  0 \le t \le T.
\end{multline*}
Meanwhile, if $\beta < \delta$, then
\begin{multline*}
  \mathbb E\|A(t)^\beta X(t)\|^2_{L_p} \le C[\mathbb E\|A(0)^\beta X_0\|^2_{L_p}
      + \mathbb E\|f\|_{\mathcal F^{\beta,\sigma}}^2\|\varphi_1\|^2_{L_p}  \\
      + \|A(0)^\delta\varphi_2\|_{L_p}^2\mathbb E |g(0)|^2t
      + t^{1+2\sigma}],
      \qquad  0 \le t \le T.
\end{multline*}
Here,  $C>0$ is some  constant, which depends on the exponents.
\end{proposition}

\begin{remark}
According to \cite[Theorem 5.2]{yagi0.5} or \cite[Theorem 16.15]{yagi}, if $A(0)$ has a bounded $H_\infty$ functional calculus, that is equivalent to the integrable condition
\begin{multline}    \label{E24}
  \int_{\varGamma} |\lambda|^{2\theta-1}
       |\langle{A(0)^{2(1-\theta)}(\lambda-A(0))^{-2}f,g}
        \rangle_{L_p \times L_q}| |d\lambda|   \\
    \le C_2 \|f\|_{L_p} \|g\|_{L_q},
    \qquad  f \in L_p(\mathcal O),\, g \in L_q(\mathcal O),
\end{multline}
for some $0 < \theta < 1$ and $C_2>0$, where $\varGamma$ is an integral contour such that 
$$\varGamma\,{:}\, \lambda = re^{\pm\varpi i},\, 0 \le r < \infty,$$
 and $\frac1p+\frac1q = 1$, then the domains of its fractional powers are given as follows. For $0 < \theta < 1$, the domains $\mathcal D(A(0)^\theta)$ are characterized by
\begin{equation*}
  \mathcal D(A(0)^\theta) = \begin{cases}
      H^{2\theta}_p(\mathcal O), \qquad 0 < \theta < \tfrac1{2p},  \\
      H^{2\theta}_{p,D}(\mathcal O), \qquad \tfrac1{2p} < \theta < 1,
             \theta\not= \tfrac{p+1}{2p}    \end{cases}
\end{equation*}
with norm equivalence. Here, $H^s_p(\mathcal O)$ denotes the Lebesgue space on $\mathcal O$ of exponent $s \ge 0$. For $s > \frac1p$, $H^s_{p,D}(\mathcal O)$ denotes a subspace of $H^s_p(\mathcal O)$ which consists of functions vanishing on the boundary $\partial\mathcal O$, i.e., 
$$H^s_{p,D}(\mathcal O) = \{u \in H^s_p(\mathcal O); u_{|\partial\mathcal O}=0\}.$$

 By the moment inequality
\begin{equation*}
  \|A(0)^{1-\theta}(\lambda-A(0))^{-1}\|_{L(E)}
  \le C_3(|\lambda|+1)^{-\theta},
  \qquad  \lambda \in \varGamma
\end{equation*}
with some $C_3>0$, 
we have an estimate for the integrand of \eqref{E24}  only that
\begin{align*}
  &|\lambda|^{2\theta-1}
       |\langle{A(0)^{2(1-\theta)}(\lambda-A(0))^{-2}f,g}
        \rangle_{L_p \times L_q}|  \\
  & \le C_3^2 |\lambda|^{2\theta-1}(|\lambda|+1)^{-2\theta}
        \|f\|_{L_p} \|g\|_{L_q}.
\end{align*}
In this sense, the integral in \eqref{E24}  is critical, and the integrable condition is not always the case.
\par

On the other hand, it is always possible to estimate $\mathcal D(A(0)^\theta)$ closely from interior. In fact,
\begin{equation*}
  \mathcal D(A(0)^\theta) \supset \begin{cases}
      H^{2\theta'}_p(\mathcal O), \qquad 0 < \theta < \theta' < \tfrac1{2p},  \\
      H^{2\theta'}_{p,D}(\mathcal O), \qquad \tfrac1{2p} \le \theta < \theta' < 1
               \end{cases}
\end{equation*}
with continuous inclusion. This can then be verified just by repeating the same arguments as in the proof of \cite[Theorem 5.2]{yagi0.5} or \cite[Theorem 16.15]{yagi} (cf. also \cite[(16.36)]{yagi}).  
\end{remark}

\section*{Acknowledgements}
The authors heartily express their gratitude to the referees of this paper for making useful and  constructive comments on the style of paper.

\end{document}